\documentclass[imslayout,noinfoline,a4paper,12pt,twoside]{imsart}
\RequirePackage[OT1]{fontenc} 
\RequirePackage{amsthm,amsmath,amssymb,natbib} 

\startlocaldefs
\theoremstyle{plain}
\newtheorem{thm}{Theorem}[section]

\newtheorem{lem}{Lemma}[section]
\endlocaldefs
\let\oldsqrt\sqrt
\def\sqrt{\mathpalette\DHLhksqrt}
\def\DHLhksqrt#1#2{%
\setbox0=\hbox{$#1\oldsqrt{#2\,}$}\dimen0=\ht0
\advance\dimen0-0.2\ht0
\setbox2=\hbox{\vrule height\ht0 depth -\dimen0}%
{\box0\lower0.4pt\box2}}
\renewcommand{\[}{\begin{eqnarray*}}
\renewcommand{\]}{\end{eqnarray*}}
\newcommand{\la}{\begin{eqnarray}}
\newcommand{\al}{\end{eqnarray}}

\renewcommand{\epsilon}{\varepsilon}
\renewcommand{\phi}{\varphi}

\newcommand{\N}{{\mathbb N}}
\newcommand{\R}{{\mathbb R}}

\newcommand{\cA}{{\mathcal A}}
\newcommand{\cX}{{\mathcal X}}

\renewcommand{\P}{{\mathbb P}}\newcommand{\E}{{\mathbb E}}  

\renewcommand{\d}{{\,\text{\rm d}}}

\newcommand{\AnnMS}{{\it Ann.\ Math.\ Statist. }}               %
\newcommand{\AnnStat}{{\it Ann. Statist. }}                     %
\newcommand{\TAMS}{{\it Trans.\ Amer.\ Math.\ Soc. }}           %
\newcommand{\Halmos}{$\quad\square$}

\begin{document}

\begin{frontmatter}


{\bf \large One optional observation inflates $\alpha$}\\
{\bf\large by $100/\sqrt{n}$   per cent}

\runtitle{Optional observations}

\begin{aug}
\author{\fnms{Lutz} \snm{Mattner}
\ead[label=e1]{mattner@uni-trier.de}}

\runauthor{Lutz Mattner}

\affiliation{Universit\"at Trier} 
 {\footnotesize\tt \jobname.tex}

\address{Universit\"at Trier\\
Fachbereich IV -  Mathematik\\ 
54286 Trier \\
Germany\\
\printead{e1}}


\end{aug}

\begin{abstract}
For one-sample level $\alpha$ tests $\psi_m$
based on independent observations $X_1,\ldots,X_m$, 
we prove an asymptotic formula  for the actual  level
of the  test rejecting if at least one of the  tests
$\psi_{n},\ldots,\psi_{n+k}$  would reject.  
For $k=1$ and usual tests at usual levels $\alpha$, the result
is approximately summarized  by the title of this paper.

Our method of proof, relying on some second order asymptotic statistics
as developed by Pfanzagl and Wefelmeyer, might also be useful for
proper sequential analysis. A simple and elementary alternative
proof is given for $k=1$ in the special case of the Gauss test. 
\end{abstract}


\end{frontmatter}



\section{Main result and examples} 
\subsection{Introduction and main result}

For a given one-sample testing problem and for every sample 
size $m$, let $\psi_m$ be a test of level~$\alpha$, based on the $m$ 
independent observations $X_1,\ldots,X_m$.
Suppose that initially $n$ observation were planned, but that these
do not lead to the desired rejection of the hypothesis. Then 
some experimenters might be tempted to collect up to $k$ further observations
$X_{n+1},\ldots,X_{n+k}$, calculating after each 
the test based on the accumulated observations,
and to declare in effect a rejection of the hypothesis at level $\alpha$ if 
$\psi_m =1$ for some $m\in \{n,\ldots,n+k\}$. 
This would of course be wrong, but by how much?  
Surprisingly this question, known in the statistical literature at 
least since the publications  of Feller  \cite{Feller.1940} 
and Robbins \cite{Robbins.1952},  is usually not addressed in 
textbooks or treatises of statistics, see Subsection~\ref{Remarks} below.

The title of the present paper gives
a somewhat rough but easy to grasp answer for the simplest
case of $k=1$, approximately valid for common values of $\alpha$ and rather
general one-sample tests based on asymptotically normal
test statistics. Theorem \ref{Main result} below
gives a mathematically precise answer also for general $k$.   
We may summarize its statistical meaning as follows: 
Even an apparently slight amount of        
optional stopping will usually inflate the nominal level 
of a test by a serious amount, such as by about 10 per cent for $n=100$
and $k=1$.

In our formulation of 
Theorem~\ref {Main result}, we think of non-randomized tests $\psi_m$ 
based on upper test statistics  $T_m$  with critical value zero,
that is, $\psi_m = {(T_m>0)}$, using the indicator notation 
$(\text{statement}) := 1\text{ or }0$ according to whether
``statement'' is true or false.
Thinking only of tests actually exhausting a given level $\alpha$,
we essentially assume that this level is attained for at least one 
distribution from the hypothesis, simultaneously for all sufficiently large
sample sizes.  Theorem~\ref {Main result} refers to such a distribution,
compare assumption~\eqref{Tn as level alpha}   below, where the above
qualifier ``essentially'' has been made precise as 
``$\mbox{}+o(1/\sqrt{n})$''.   
Unfortunately this assumption already excludes lattice cases like 
the binomial tests, for which any analogue of   Theorem~\ref {Main result}
would presumably look  more complicated.
Now the test rejecting  if at least one of the tests $\psi_n,\ldots,\psi_{n+k}$ 
would  reject is   $(\max_{m=n}^{n+k} T_m >0)$, and hence, 
with respect to a given distribution of $X_1$, its probability of rejecting 
is   $\alpha_{n,k}$ as defined in \eqref{Def alpha n k} below.
Our regularity assumptions 
\eqref{StochasticExpansion}--\eqref{Rm strongly small} on the sequence $(T_n)$ 
are similar to those imposed by Pfanzagl and Wefelmeyer 
in their well-known treatise of second order asymptotic statistics,
see in particular  \cite[Section~10.3]{PW.II},
on which our result
is based. In Subsection~\ref{Remark on assumptions} 
below we comment on some minor differences between these  
assumptions. Let $\Phi$  and $\phi = \Phi'$
denote distribution function and density
of the standard normal distribution  $\mathrm{N}^{}_{0,1}$, and let us put 
\[
 h(\alpha) &:=& \frac
 {\phi\big(\Phi_{}^{-1}(1-\alpha)\big)}{\alpha\sqrt{2\pi}}
 \qquad\quad (\alpha \in \,]0,1[) 
\]
We write $A^2_{\neq}:= \{(x,y)\in A^2\,:\,x\neq y\}$ for any set $A$
and $x_+ := x\vee 0 = \max\{x,0\} = (-x)_-$ for $x\in \R$. 
Proofs of Theorem \ref{Main result} and Lemmas \ref{Lemma on h} 
and \ref{Lemma on Rn}
are given in Section \ref{Proofs}, see 
Subsections~\ref{Main proof.}, \ref{Proof on h}, and \ref{Proof on Rn}.
We point out that Example~\ref{ExGauss} below
contains an elementary direct proof, suitable for inclusion 
in standard statistics courses, of 
Theorem~\ref{Main result} in its simplest special case of 
the Gauss test with one optional observation, that is, $k=1$.

\begin{thm} \label{Main result}
Let $\cX$ be a measurable space, $(X_n)_{n\in\N}$  a sequence
of independent and identically distributed $\cX$-valued random variables,
$\alpha \in {]0,1[}$, and $(t_n)_{n\in\N}$ a sequence of measurable
functions $t_n:\cX^n \rightarrow \R $ such that the random variables
\[
  T_n &:=& t_n(X_1,\ldots,X_n) \qquad\quad (n\in\N)
\]  
satisfy  
\la                    \label{Tn as level alpha}
  \P (T_n>0) &=& \alpha + o(\frac 1{\sqrt{n}}) 
                 \qquad\quad (n\rightarrow\infty)
\al
Let  
\la   \label{Def alpha n k}
  \alpha_{n,k}  \,:=\,  \P (\max_{m=n}^{n+k} T_m>0),&&
  \rho_{n,k}  \,:=\, \frac{\alpha_{n,k}}\alpha -1  
          \qquad\quad (n,k\in \N)
\al
Assume that for $n\in \N$ 
\la  \label{StochasticExpansion}
  T_n  &=&  \phantom{+\,}\,
  \mu_0   + \frac 1{\sqrt{n}}\sum_{i=1}^n {f}_0(X_i)\\
\nonumber  && 
     +\, \frac 1{\sqrt{n}} \Big(
  \mu_1 +  \frac 1{\sqrt{n}}\sum_{i=1}^n {f}_1(X_i) 
       +\frac 1{2n} \sum_{(i,j)\in\{1,\ldots,n\}^2_{\neq}}f_2(X_i,X_j)\Big)\\
\nonumber  &&   +\,  R_n
\al  
for some constants $\mu_0,\mu_1\in\R$, measurable functions
$f_0,f_1:\cX\rightarrow \R$ and $f_2:\cX^2 \rightarrow \R$,
and a sequence $(R_n)_{n\in \N}$ of real-valued random variables
with
\la
 &&\E f_0(X_1) = \E f_1(X_1)=\E \big(f_2(X_1,X_2) {\pmb|}X_1  \big)=0,\,\, 
  f_2(X_1,X_2)=f_2(X_2,X_1)   \label{centeredness etc} \\
 &&  \E \big(f_0(X_1)\big)^2 = 1, \quad   \E|f_0(X_1)|^{3}_{} <\infty,\quad  
  \text{\rm $f_0(X_1)$ has a non-lattice law}   \label{Assumptions on f0}  \\
 &&  \E|f_1(X_1)|_{}^{3/2} <\infty, \quad 
  \E|f_2(X_1,X_2)|^{2+\delta}_{} <\infty \text{ \rm for some $\delta>0$} 
      \label{moments finite}\\
 &&  \label{SecondCondOnRn}
\text{\rm For every $\epsilon>0$: }\quad   \sup_{t\ge 1} t\,\P( |R_n|> \frac
 {t\epsilon}{\sqrt{n}}) = o(\frac 1{\sqrt{n}}) 
  \qquad\quad (n\rightarrow\infty)    \label{Rm strongly small} 
\al
Then
\la                  \label{MainResultII}
 \quad \rho_{n,k}
 &=& 
 \frac {h(\alpha)}{\sqrt{n}}\sqrt{2\pi} \sum_{\ell=1}^k\frac 1{\ell}
  \E \Big( \sum_{i=1}^\ell  {f_0(X_i)}\Big)_+
   \,+\, o\Big(\,\sqrt{\frac kn} \,\Big) 
\qquad (\frac kn\rightarrow 0)
\al
and
\la                 \label{MainResultIII}
 \rho_{n,k}
 &=&   2\, h(\alpha)\sqrt{\frac kn} \,+\, o\Big(\,\sqrt{\frac kn} \,\Big) 
\qquad\quad (\frac kn\rightarrow 0,\ k\rightarrow \infty)
\al
\end{thm}

For common levels $\alpha$, we have $h(\alpha)\approx 1$:
\begin{lem}\label{Lemma on h} 
The function $h$ is strictly decreasing with
the asymptotic behaviour 
\la          \label{asymp h}
  h(\alpha) &\sim&  \sqrt{ \frac 1\pi \log(\frac 1\alpha)}
 \qquad (\alpha \rightarrow 0)
\al 
and rounded values  
\[
\begin{array}{|c||c|c|c|c|c|c|c|c|c|c|}
\hline
\alpha   &0.05 & 0.025  &0.01 & 0.005 & 0.001  & 0.0005   \\
\hline
h(\alpha)&0.82 & 0.93 & 1.06 & 1.15 &  1.34&  1.42 \\
\hline
\end{array}
\]
\end{lem}

Taking  $k=1$ in \eqref{MainResultII}, we get 
\la
 \rho_{n,1}&\sim&\frac{h(\alpha)}{\sqrt{n}}\sqrt{2\pi}\,
 \E\big({f_0(X_1)}\big)_+ \qquad \quad(n\rightarrow\infty)
\al 
So, assuming $h(\alpha)\approx 1$, the claim in the title of this paper
approximately results when $\sqrt{2\pi}\, \E\big({f_0(X_1)}\big)_+ \approx 1$
and $n$ is sufficiently large. For the Gauss and $t$-tests 
in Examples~\ref{ExGauss} and \ref{Ex t test} below,  
we have $\sqrt{2\pi}\, \E\big({f_0(X_1)}\big)_+=1$ exactly.
In theses two cases, from \eqref{MainResultIIGauss} below, 
one optional observation inflates $\alpha$ by $h(\alpha)100/\sqrt{n}$
per cent,  two optional observation inflate $\alpha$ by 
$h(\alpha)171/\sqrt{n}$ per cent, etc.
For general examples we note, using 
\eqref{centeredness etc}, \eqref{Assumptions on f0} and 
$\E |Y|\le ( \E Y^2 )^{1/2}$, 
that  $\sqrt{2\pi}\, \E\big({f_0(X_1)}\big)_+$ can be any strictly 
positive number $\le\, \sqrt{2\pi}\frac 12  
\big( \E\big({f_0(X_1)}\big)^2\big)^{1/2} = \sqrt{\pi/2}$,
so the accuracy of the 
claim in the title depends on   $\sqrt{2\pi}\, \E\big({f_0(X_1)}\big)_+$
being not too far from its value under $f_0(X_1)\sim \mathrm{N}_{0,1}$.
In the exponential Example \ref{ExampleExp}, we have   
$\sqrt{2\pi}\, \E\big({f_0(X_1)}\big)_+ = \sqrt{2\pi}/\mathrm{e}
= 0.92$,  so that in this case one optional observation inflates $\alpha$ by merely 
$h(\alpha)92/\sqrt{n}$ per cent. 

Many test sequences $(\psi_n)$ in the literature can be written in the
form $\psi_n = (T_n>0)$ with $(T_n)$ admitting an expansion 
as in Theorem~\ref{Main result}. This is in particular true,
under appropriate regularity conditions, for one-sided tests
based on one-dimensional components of minimum contrast
estimators, see \cite[pp.~395-396, Theorem 11.3.4]{PW.II}
for a precise statement and references.
In our examples in Subsection~\ref{Example} below  we can easily 
check all assumptions rather directly. 

We have to note here that our assumption \eqref{SecondCondOnRn}
on the sequence of remainders $(R_n)$ is 
slightly stronger than Pfanzagl and Wefelmeyer's 
\la  \label{Rm weakly small}
&& \text{\rm For every $\epsilon>0$: }\quad  
    \P(|R_n| > \epsilon/\sqrt{n})=o(1/\sqrt{n})
 \qquad\quad (n\rightarrow\infty)   
\al
Condition  \eqref{SecondCondOnRn} 
appears to be just about what is needed in the proof 
of our crucial Lemma \ref{Main Lemma}  below, since we  
allow $k$ to be unbounded, see \eqref{Bounding U_5} below.
For  bounded $k$,  assumption  \eqref{Rm weakly small}
would suffice. Condition  \eqref{SecondCondOnRn} should be easy to establish
in any reasonable case, and we do this in 
Example~\ref{Ex t test} below by using the following simple fact.

\begin{lem}\label{Lemma on Rn}
Let $(R_n)_{n\in\N}$ be a sequence of real-valued random variables
such that for some $p\in[1,\infty[$ and $n_0\in\N$ the random variables 
\la                                       \label{Def Yn}
 Y_n &:=& |n^{\frac {1+p}{2p}} R_n|^p 
\al
with $n\ge n_0$ are uniformly integrable.
Then \eqref{SecondCondOnRn} holds.
\end{lem}

For a discussion of further minor differences  between 
Pfanzagl and Wefelmeyer's  and our assumptions on $(T_n)$ see 
Subsection~\ref{Remark on assumptions} below.

\subsection{Various remarks} \label{Remarks} 
Reading this subsection is  not logically necessary for understanding the 
rest of this paper. 

Under the assumptions of Theorem \ref{Main result},
one can easily show that  $T_n-\mu_0 - (1/\sqrt{n})\sum_{i=1}^n f_0(X_i)$
converges to zero in probability, so that 
$T_n-\mu_0 $ converges in law  to $\mathrm{N}_{0,1}$, and hence, in view
of \eqref{Tn as level alpha}, we must have
\la
  \mu_0 &=& -\Phi^{-1}(1-\alpha)
\al 
The result of \cite[Corollary 10.3.8]{PW.II},
on which our proof of Theorem~\ref{Main result}
will be based below, further includes a formula for $\mu_1$
in terms of $\alpha,f_0,f_1,f_2$ and the law of $X_1$.

The expectations occuring in formula \eqref{MainResultII}
can be computed explicitly in some cases, see in particular 
Example~\ref{ExampleExp} below and, more generally,
\cite{Diac.Zab.1991}. We always have 
$ \E (\sum_{i=1}^\ell f_0(X_i))_+ \sim \sqrt{\frac {\ell}{2\pi}}$ 
for $\ell\rightarrow \infty$, see the end of the proof of 
Theorem \ref{Main result}, and 
$   \E (\sum_{i=1}^\ell f_0(X_i))_+ \ge \sqrt{\frac {\ell}{2}} \,
\E(f_0(X_1))_+$ by \cite[Corollary 1.3]{Mattner.2003}.

Relation \eqref{MainResultIII}
becomes false  if the condition   ``$k/n\rightarrow 0$''
is replaced by ``$k/n$ bounded'', since for $k/n$ constant and sufficiently
large  a contradiction to $\alpha_{n,k}\le 1$ would result.

As mentioned 
above, the problem of level inflation due to optional stopping is
usually  not addressed   in textbooks or  treatises of statistics.
It was raised, perhaps for the first time in the literature, 
by Feller in 1940
in connection with apparently ill-conducted experiments concerning 
``extra-sensory perception'', see \cite[pp.~286-294]{Feller.1940}
and references therein. Robbins \cite[pp.~534-535]{Robbins.1952}
posed the problem of evaluating or bounding 
what we have called $\alpha_{n,k}$, and stated without proof  
a bound in the case of the Gauss test. We are not aware 
of a continuation of that part of Robbins' work.
Diaconis \cite{Diaconis.1978} comments critically on  
Feller's paper, but not on the particular point of optional stopping.
Among books known to the present author, 
Pfanzagl's \cite[p.~127)]{Pfanzagl.1994}
is unique in stressing and demonstrating
the problem, albeit only by  a simulation, and unfortunately
obscured by the additional deliberate mistake 
of choosing between two valid test for each sample size.
To our surprise, we did not find any statistical textbook treating 
the problem more systematically.

Theorem~\ref{Main result}  can be read as addressing 
an improper sequential analysis.
Its technical basis however, namely  
the consideration of statistics  with stochastic expansion,
Pfanzagl and Wefelmeyer's result on their asymptotic distributions, 
and the crucial  Lemma~\ref{Main Lemma} below, 
might be useful for proper sequential analysis as well.
For example, under the assumptions of Theorem~\ref{Main result}
but with condition \eqref{Tn as level alpha} omitted,
we can generalize \eqref{MainResultII} to a computation 
of the asymptotic distribution of $\max_{m=n}^{n+k} T_m$ 
up to an error $o(\sqrt{k/n})$ for $k/n\rightarrow 0$,
by using \cite[Proposition 10.3.1]{PW.II}
rather than \cite[Corollary 10.3.8]{PW.II}
in a modification of the present Proof~\ref{Main proof.}.

\subsection{Examples} \label{Example}
In each case below, let $\alpha\in{]0,1[}$ be  the level of the tests
considered.  
\subsubsection{The Gauss test} \label{ExGauss}
The Gauss test for testing  $\mu\le \mu_0$ based
on i.i.d.\ normal $X_1,\ldots,X_n$ with unknown mean $\mu \in \R$
and known standard deviation  $\sigma_0\in{]0,\infty[} $ rejects iff
\[
  T_n &:=& - \Phi^{-1}(1-\alpha) + \frac1{\sqrt{n}} 
 \sum_{i=1}^n\frac{X_i-\mu_0}{\sigma_0} \,\, >\,\,0
\] 
Hence Theorem~\ref{Main result} is applicable, with 
$X_i \sim \mathrm{N}_{\mu_0,\sigma_0^2}$,
$f_0(x) := (x-\mu^{}_0)/\sigma_0$, 
and vanishing $\mu_1$, $f_1$, $f_2$, and $R_n$,
and  \eqref{MainResultII} reads 
\la        \label{MainResultIIGauss}
 \rho_{n,k} & =& \frac{h(\alpha)}{\sqrt{n}} \sum_{\ell=1}^k\frac
 1{\sqrt{\ell}}  \,+\, o\Big(\,\sqrt{\frac kn} \,\Big) 
\qquad \quad(\frac kn\rightarrow 0)
\al

Here is the elementary proof of \eqref{MainResultIIGauss}
for the simplest  case of $k=1$ promised 
immediately before the statement of Theorem~\ref{Main result}:
With $Y_i := (X_i-\mu_0)/\sigma_0$,
$Z_n := \frac 1{\sqrt{n}}\sum_{i=1}^n Y_i $
and $z:=\Phi^{-1}_{}(1-\alpha)$, we have
\[
 \alpha_{n,1}-\alpha 
 &=&\nonumber \P(Z_n>z\text{ or }Z_{n+1}>z)-\P(Z_n>z)  \\
 &=&\nonumber \P(Z_n\le z,\, Z_{n+1}>z) \\
 &=&\nonumber\P(Z_n\le z,\, Y_{n+1}>\sqrt{n+1}z  - \sqrt{n}Z_n) \\
 &=& \int_{-\infty}^z \big(1-\Phi(\sqrt{n+1}z-\sqrt{n}t) \big)
     \phi(t)\d t
\]
since $Z_n$ and $Y_{n+1}$ are independent 
and $\mathrm{N}^{}_{0,1}$-distributed. Hence, using the change of variables
$t\mapsto z- \frac {t}{\sqrt{n}}$, we get 
\[
\sqrt{n} \,\big(\alpha_{n,1} -\alpha \big) 
 &=&\nonumber\int_0^\infty\Big(1-\Phi\big((\sqrt{n+1}-\sqrt{n})z +t\big) \Big)
     \phi(z- \frac{t}{\sqrt{n}})\d t \\
 &\underset{(n\rightarrow\infty)}{\longrightarrow}&
  \int_0^\infty \big(1-\Phi(t)\big) \phi(z) \d t  \\
 &=&\nonumber\phi(z) \int_0^\infty t\phi(t) \d t  \\
&=&\nonumber \frac {\phi(z)}{\sqrt{2\pi}}
\]
by dominated convergence with the integrands dominated by
the function $t\mapsto \big(1-\Phi(t-z_-)\big)\phi(0)$.

\subsubsection{Testing an exponential mean}\label{ExampleExp}
The usual optimal test for $\lambda \ge \lambda_0$ based on 
i.i.d.\ exponential $X_1,\ldots,X_n$ with density
${]0,\infty[} \ni x\mapsto \lambda\mathrm{e}^{-\lambda x}$ with
$\lambda \in{]0,\infty[}$ unknown
rejects for large values of $\sum_{i=1}^n X_i$, namely iff
\[
  T_n &:=& - F^{-1}_{P_n}(1-\alpha) 
   + \frac1{\sqrt{n}}\sum_{i=1}^n(\lambda_0X_i-1)\,\,>\,\,0
\] 
where $P_n$ denotes the law of the standardization of $\sum_{i=1}^n X_i$
under $\lambda = \lambda_0$
and $F^{-1}_{P_n}$ the corresponding quantile function. Since   
$P_n$  admits an Edgeworth expansion with remainder $o(1/\sqrt{n})$,
see e.g.\ \cite[p.~174, Theorem 5.22]{Petrov.1995},
we have 
\[
 T_n &=& -\Phi^{-1}(1-\alpha) + \frac1{\sqrt{n}}\sum_{i=1}^nf_0(X_i)
     + \frac{ \mu_1}{\sqrt{n}}    + R_n
\]
where $f_0(x)=\lambda_0 x-1$, $\mu_1\in \R$ depends only on $\alpha$, 
and  where $R_n$ is deterministic and $o(1/\sqrt{n})$ for $n\rightarrow\infty$.
Hence  the assumptions of Theorem~\ref{Main result} are fulfilled.
With $\gamma_a(x) := \big(\Gamma(a) \big)^{-1} x^{a-1}\mathrm{e}^{-x}$
we have $\int_x^\infty (t-a)\gamma_{a}(t)\d t =a\,\gamma_{a+1}(x)$ 
for $a,x\in{]0,\infty[}$, by differentiation with respect to $x$
and considering $x\rightarrow\infty$, so that
\[
 \E\Big( \sum_{i=1}^\ell  {f_0(X_i)}\Big)_+
 &=& \int_\ell^\infty (t-\ell)\gamma_\ell(t)\d t 
 \,\,=\,\,\Big( \frac{\ell}{\mathrm{e}}\Big)^{\ell}_{}
 \frac 1{(\ell-1)!}  
\] 
and accordingly \eqref{MainResultII} reads
\la \qquad
 \rho_{n,k} & =& \frac{h(\alpha)}{\sqrt{n}} 
   \frac {\sqrt{2\pi}}{\mathrm{e}} \sum_{\ell=1}^k 
       \big(\frac \ell {\mathrm{e}}\big)_{}^{\ell-1} \frac 1{(\ell-1)!}   
 \,+\, o\Big(\,\sqrt{\frac kn} \,\Big) \qquad \quad(\frac kn\rightarrow 0)
\al

\subsubsection{The $t$-test} \label{Ex t test}
The $t$-test for $\mu\le \mu_0$ based  
on i.i.d.\ normal $X_1,\ldots,X_n$ with unknown mean $\mu \in \R$
and unknown standard deviation  $\sigma\in{]0,\infty[} $ rejects 
for $n\ge 2$ iff
\[
  T_n &:=& \frac1{\sqrt{n}}\sum_{i=1}^n Y_i - 
     c_n\,\sqrt{\frac1{n-1}\sum_{i=1}^n(Y_i-\overline{Y}_n)_{}^2} 
  \,\,>\,\,0
\]
with $Y_i:=(X_i-\mu_0)/\sigma_0$ with $\sigma_0\in{]0,\infty[}$
arbitrary, $ \overline{Y}_n=\frac 1n\sum_{i=1}^n Y_i$,
and $c_n$ denoting the $(1-\alpha)$-quantile of the 
$t$-distribution with $n-1$ degrees of freedom. 
Assuming $X_i\sim\mathrm{N}_{\mu_0,\sigma_0^2}$ from now on, 
the $Y_i$ are standard normal.
We  write  $ \overline{Y^2_n} := \frac 1n \sum_{i=1}^n Y_i^2$ 
and 
$S_n^2:= \frac 1n \sum_{i=1}^n (Y_i-\overline{Y}_n)^2 
 =  \overline{Y^2_n} - \overline{Y}_n^2$ 
and use $ c_n = \Phi^{-1}(1-\alpha) + O(\frac 1n)$
e.g.\ from \cite[p.~461, (11.75)]{Lehmann.Romano.2005},
$\sqrt{\frac{n}{n-1}}=1 + O(\frac 1n)$ for $n\in \N$ with $n\ge 2$,
and $\sqrt{x} = 1+ (x-1)/2 + O((x-1)^2)$ for $x\in [0,\infty[$, 
to get   
\[
 T_n &=& \frac1{\sqrt{n}}\sum_{i=1}^n Y_i \,-\,  
  c_n \sqrt{\frac{n}{n-1}} \sqrt{S_n^2   }    \\
&=&\frac1{\sqrt{n}}\sum_{i=1}^n Y_i  \\
 && \, -\, \Big( \Phi^{-1}(1-\alpha) + O(\frac 1n) \Big)
    \Big(1+\frac12 (\overline{Y^2_n} -1) - \frac12\overline{Y}_n^2
      + O\big( (S_n^2-1)^2 \big) \Big)  \\
    &=&  - \Phi^{-1}(1-\alpha) \,+\, \frac1{\sqrt{n}}\sum_{i=1}^n Y_i
      \,-\,\frac{ \Phi^{-1}(1-\alpha) }{2n}\sum_{i=1}^n (Y_i^2 -1)\,+\,R_n      
\]
where the sequence $(R_n)$ satisfies  \eqref{SecondCondOnRn}
by Lemma \ref{Lemma on Rn} with $p=1$,
since $nR_n$ is a linear combination with bounded 
coefficients of the four random variables $1$,
$\frac 1n\sum_{i=1}^n (Y_i^2 -1)$,  $n \overline{Y}_n^2$,  and $n(S_n^2-1)^2$,
which are uniformly integrable, as may  be verified 
by checking that  their second moments are bounded. 
Hence the assumptions of Theorem \ref{Main result} 
are fulfilled, with $f_0(x) = (x-\mu_0)/\sigma_0$,
$\mu_1=0$, $f_1(x)= - \Phi^{-1}(1-\alpha)( ((x-\mu_0)/\sigma_0)^2-1)$, 
and $f_2=0$,  and we get  the same asymptotic formula 
\eqref{MainResultIIGauss} as in the Gauss case.

\section{Auxiliary results and proofs} \label{Proofs} 
In this section we use Pfanzagl and Wefelmeyer's \cite[p.~16]{PW.II}
$\epsilon^{}_\P$-notation:
For real-valued random variables $X_n$ on  probability spaces
$(\Omega_n,\cA_n,\P_n)$ and numbers $\delta_n >0$, we write
\[
 && X_n \,=\,\epsilon^{}_{\P_n}(\delta_n) \qquad(n\rightarrow\infty)\\
  &:\iff& \forall \epsilon >0 \quad \P_n(|X_n|\ge \epsilon)\,=\,o(\delta_n) 
  \qquad(n\rightarrow\infty)
\]  
Here $n$ can belong to any index set if ``$n\rightarrow\infty$''
is replaced by the specification of some appropriate passage to the limit,
formally by a filter or a net. 
In our case  the index is actually $(n,k)\in\N^2$,
but $\P_{n,k}$ is for notational convenience
chosen to be independent of $(n,k)$, say 
an infinite product measure,
so that $\epsilon^{}_{\P_{n,k}}$ becomes $\epsilon^{}_{\P}$. 
The three successively more specialized 
passages  to the limit we use are 
``$n\rightarrow \infty$'',
``$k/n \rightarrow 0$'', and 
``$k/n \rightarrow 0, k\rightarrow \infty$''.

We begin with a comparison of our versus Pfanzagl and Wefelmeyer's
assumptions on the stochastic expansion
\eqref{StochasticExpansion}, then state and prove 
the  crucial  Lemma~\ref{Main Lemma}, and conclude by proving 
Theorem~\ref{Main result} and Lemmas~\ref{Lemma on h} and \ref{Lemma on Rn}.

\subsection{Discussion of the assumptions
\eqref{StochasticExpansion}, \eqref{Assumptions on f0}, 
and \eqref{SecondCondOnRn}}  \label{Remark on assumptions} 
Our assumptions on the sequence $(T_n)$ differ in three respects from those 
of \cite[p.~343, Corollary 10.3.8, $S_n= \mu(P) + T_n/\sqrt{n}$, 
the case $g_1=\overline{g}_1=0$]{PW.II}
used in the  proof of Theorem~\ref{Main result} below.

First,  to  simplify  the notation, we have added the normalizing assumption
$\E (f_0(X_1))^2 =1$.

Second, as already discussed above, we have \eqref{SecondCondOnRn} 
instead  of Pfanzagl and Wefelmeyer's ``$\sqrt{n}\,R_n =\epsilon_\P(1/\sqrt{n})$'', 
that is, \eqref{Rm weakly small}.

Third, instead of our  
$U_n :=\frac 1{2n}\sum_{(i,j)\in\{1,\ldots,n\}^2_{\neq}}f_2(X_i,X_j)$
in the stochastic expansion of $T_n$, in \cite{PW.II} we have
$V_n := \frac 1{2n}\sum_{(i,j)\in\{1,\ldots,n\}^2} f_2(X_i,X_j)$
and the additional assumption  $\E (f_2(X_1,X_1))^{3/2} < \infty$. 
Here our version is slightly more general,
since, under the moment condition just stated, we have 
$ V_n = U_n + \frac 12  \E f_2(X_1,X_1) +  \sqrt{n} \, R_n$,
where the present $R_n := n^{-3/2}\sum_{i=1 }^n \xi_i$
with $\xi_i := \frac 12(f_2(X_i,X_i) - \E f_2(X_1,X_1))$ 
also satisfies \eqref{SecondCondOnRn},
as follows via Lemma \ref{Lemma on Rn}
from the fact that for $p:=3/2$ the random 
variables  $Y_n$ defined by \eqref{Def Yn} 
are given by $Y_n =  |n^{-2/3} \sum_{i=1}^n \xi_i |_{}^{3/2} $
and hence are uniformly integrable by the Theorem of 
Pyke and Root \cite{Pyke.Root.1968}.
Hence, even under our more stringent condition  \eqref{SecondCondOnRn}
on the remainders, we may in 
the expansion from  \cite{PW.II} simultaneously
replace $V_n$ by $U_n$ and  $\mu_1$
by $\mu_1 + \frac 12  \E f_2(X_1,X_1)$.
Moreover, \cite[Proposition 10.3.1 and  Corollary 10.3.8]{PW.II}
remain true with $U_n$ in place of $V_n$ even if the assumption
 $\E (f_2(X_1,X_1))^{3/2} < \infty$ is omitted, 
since the latter is used in \cite{PW.II} only to replace $V_n$ by $U_n$ in 
the proof of \cite[Proposition  10.3.1]{PW.II} in order
to prepare for the application of the result 
of Bickel, G\"otze  and van Zwet 
\cite[Theorem~1.2]{BGvZ.1986}
and  G\"otze  \cite[Theorem~1.14]{Goetze.1987}
which refers to $U$-statistics rather than $V$-statistics. 
 
Finally let us note that our non-latticeness assumption 
in \eqref{Assumptions on f0} is the same as the one imposed in \cite{PW.II}
using the  confusing term ``strongly non-lattice''
necessary only for multivariate statistics  $T_n$, 
see \cite[pp.~207, 221, and 226]{BRR.1986}
\subsection{The main lemma}
The following Lemma~\ref{Main Lemma} is the crucial first
step in our proof of Theorem~\ref{Main result} in 
Subsection~\ref{Main proof.} below. 

\begin{lem} \label{Main Lemma}
Let $(T_n)_{n\in\N}$ be a sequence of real-valued random
variables such that for $n\in\N$ we have 
\eqref{StochasticExpansion}
for some constants $\mu_0,\mu_1\in\R$, 
a measurable space $\cX$,
a sequence $(X_n)_{n\in\N}$ of independent and identically distributed
$\cX$-valued random variables,
measurable functions $f_0,f_1:\cX\rightarrow \R$ and $f_2:\cX^2 \rightarrow \R$,
and a sequence $(R_n)_{n\in \N}$ of real-valued random variables
with \eqref{centeredness etc}, 
\la
&& \E|f_0(X_1)|^{2}_{} <\infty, \quad 
  \E|f_2(X_1,X_2)|^{\frac 32+\delta}_{} <\infty  
   \text{ \rm for some $\delta>0$} \label{lower moments finite}
\al
and \eqref{Rm strongly small}.
Then
\la             \label{Strong expansion}
 \quad \max_{m=n}^{n+k} \Big| T_m 
  -\big( T_n +\frac 1{\sqrt{n}} \sum_{i=n+1}^mf_0(X_i)\big) \Big|
      &=&  \sqrt{\frac kn} \epsilon^{}_{\P}\Big(\, \sqrt{\frac kn}\,\Big)  
  \qquad  (\frac kn\rightarrow 0)
\al
\end{lem}

\begin{proof} Let $\epsilon >0$.
We have for $n,k\in \N$
\[
U &:=& \text{L.H.S.\eqref{Strong expansion}} \\
&\le&
 \max_{m=n}^{n+k}\big|(\frac1{\sqrt{m}}- 
      \frac1{\sqrt{n}})\sum_{i=1}^mf_0(X_i)\big| 
 \,+\, \max_{m=n}^{n+k}\big| (\frac1{\sqrt{m}}-\frac1{\sqrt{n}})\mu_1 \big| \\
&& \,+\,\max_{m=n}^{n+k}\big| \frac 1m\sum_{i=1}^mf_1(X_i) - 
 \frac 1n\sum_{i=1}^n f_1(X_i) \big| \\
&& \,+\, \max_{m=n}^{n+k}\big| 
 \frac 1{m^{3/2}}
\sum_{j=2}^{m}\sum_{i=1}^{j-1}
 f_2(X_i,X_j) -  \frac 1{n^{3/2}}
\sum_{j=2}^{n}\sum_{i=1}^{j-1} f_2(X_i,X_j)  \big|   \\ 
&&  \,+\, \max_{m=n}^{n+k}\big| R_m -R_n \big| \\
&=:& U_1 + U_2 + U_3 + U_4 + U_5 
\]
For $\alpha\in \R$, the elementary inequality
\[
  (1-x)^\alpha &\ge& 1-(\alpha \vee 1 )\, x \qquad\quad(x\in[0,1[) 
\] 
applied to $x=k/(n+k)$ yields
\la  \label{AuxInq n k alpha}\quad
  \frac 1{n_{}^\alpha} -\frac 1 {(n+k)^\alpha} 
 &=& \frac  { 1-(1-\frac k{n+k})_{}^\alpha}{ {n_{}^\alpha}}
 \,\, \le \,\,
\frac {(\alpha \vee 1 )\,k}{n_{}^\alpha\, (n+k)} 
  \,\, \le\,\,  \frac {(\alpha \vee 1 )\,k}{n_{}^{1+\alpha} } 
\al

Investigation of $U_1$: By \eqref{AuxInq n k alpha} with $\alpha=1/2$,
\[
  U_1 &\le &  \frac{k}{\sqrt{n}\,(n+k) } 
   \max_{m=n}^{n+k}\big|\sum_{i=1}^mf_0(X_i)\big| 
\]
and by Kolmogorov's inequality, see \cite[p.~61]{Durrett.2005},
\[
 \P( U_1 > \epsilon \sqrt{ \frac kn}) 
  &\le& \epsilon^{-2}\, \E\big(f_0(X_1)\big)^2\, \frac {k}{n+k}
 \,\, =\,\, o\Big(\, \sqrt{\frac kn}\,\Big) 
   \qquad\quad( \frac kn \rightarrow 0)
\]

Investigation of $U_2$: Again by  \eqref{AuxInq n k alpha}
with $\alpha=1/2$, we get
\[
 U_2  &\le&  \frac{|\mu_1|k}{\sqrt{n}(n+k)} 
 \,\,= \,\,  o\Big(\, \sqrt{\frac kn}\,\Big)  \qquad\quad (n\rightarrow \infty)
\]

Investigation of $U_3$:
By \eqref{AuxInq n k alpha} with $\alpha=1$, we get 
\[
 U_3 
  &=&\max_{m=n}^{n+k} \Big|(\frac 1m-\frac 1n) \sum_{i=1}^n f_1(X_i) 
               \,+\,  \frac 1m  \sum_{i=n+1}^m f_1(X_i)     \Big| \\
  &\le&  \frac k{n^2} 
   \Big|\sum_{i=1}^n f_1(X_i)\Big| 
   \,\,+\,\, \frac 1n \max_{m=n}^{n+k} \Big|  \sum_{i=n+1}^m f_1(X_i)  \Big|\\
  &=:& U_{3,1} + U_{3,2}
\]
By Markov's inequality and the L$^1$-law of large numbers,
see \cite[p.~337]{Durrett.2005},
\[
 \P\Big( |U_{3,1}| > \epsilon\sqrt{\frac kn}\,\Big)   
 &\le&\frac 1{\epsilon}\sqrt{\frac kn}\, \E\big|\frac 1n \sum_{i=1}^n f_1(X_i)\big| 
 \,\, =\,\,o\Big(\, \sqrt{\frac kn}\,\Big) 
 \qquad (n\rightarrow\infty)
\]
By Doob's inequality applied to the submartingale 
$(|\sum_{i=1}^\ell f_1(X_i)| \,:\,
\ell\in\{0,\ldots,k\})$, see \cite[p.~247]{Durrett.2005},
and recalling our indicator notation 
$(\text{statement}) := 1\text{ or }0$ according to whether
``statement'' is true or false,  
\la \nonumber
 \P\Big( |U_{3,2}| > \epsilon \sqrt{ \frac kn } \,\Big) 
  &=&  
 \P\big( \max_{\ell=1}^k  \big|\sum_{i=1}^\ell f_1(X_i)\big|
   > \epsilon\sqrt{nk}\,\big) \\
 &\le&   \label{U32}
    \frac 1\epsilon \sqrt{\frac kn}\,
     \E\Big( \big| \frac 1k \sum_{i=1}^k f_1(X_i)\big|
   \cdot\big(  |\max_{\ell=1}^k \sum_{i=1}^\ell f_1(X_i)|> \epsilon
    \sqrt{nk} \,   \big)\Big)  \\
 &=&\nonumber  o\Big(\, \sqrt{\frac kn}\,\Big) \qquad\quad (n\rightarrow\infty)
\al
where for the last step, given $\delta>0$, we choose $k_0$ according
to the L$^1$-law of large numbers such that 
$\E \big| \frac 1k \sum_{i=1}^k f_1(X_i)\big| <\delta/\epsilon$
for $k>k_0$, and then $n_0$ such that for $k\le k_0$ and $n\ge n_0$
the expectation in  line \eqref{U32} is $ <$ $\delta/\epsilon$.

Investigation of $U_4$: By \eqref{AuxInq n k alpha} with $\alpha=3/2$,
we get 
\[
 U_4 &=& \max_{m=n}^{n+k}\Big| (\frac 1{m_{}^{3/2}}- \frac1{n_{}^{3/2}})  
\sum_{j=2}^{n}\sum_{i=1}^{j-1}    f_2(X_i,X_j) 
 \,+\, \frac 1{m_{}^{3/2}}\sum_{j=n+1}^m\sum_{i=1}^{j-1} f_2(X_i,X_j) \Big|\\
 &\le& \frac{3k}{2\,n_{}^{3/2}(n+k)} \Big|\sum_{j=2}^{n}\sum_{i=1}^{j-1}
   f_2(X_i,X_j)\Big|
   \,\,+\,\,\frac 1{n_{}^{3/2}}\max_{m=n}^{n+k}\Big|\sum_{j=n+1}^m\sum_{i=1}^{j-1} 
    f_2(X_i,X_j)\Big|\\
&=:&   U_{4,1} \,+\, U_{4,2} 
\]
Let $p:= (\frac 32 +\delta)\wedge 2$ for some $\delta$ from 
\eqref{lower moments finite}. With 
$c_1 :=  2\,(\frac 3{2\epsilon})^p \E|f_2(x_1,X_2)|^p <\infty$,
Markov's inequality and inequality \eqref{vBE for U}
from Lemma \ref{Lem vBE for U} below applied to $f_{ij} := f_2$ yield
\[
 \P\Big( |U_{4,1}| >\epsilon  \sqrt{\frac kn}\,      \Big)
 &\le& \big(\frac 3{2\epsilon} \big)^p_{}k^{p/2}n^{-p}(n+k)^{-p} 
    \,\E \big|\sum_{j=2}^{n}\sum_{i=1}^{j-1} f_2(X_i,X_j)  \big|_{}^p \\
 &\le & c_1\, k^{p/2}n^{2-p}(n+k)^{-p}   \\ 
 &=& c_1 \sqrt{\frac kn}\, k^{(p-1)/2} n^{5/2-p}(n+k)^{-p}  \\
 &\le&   c_1 \sqrt{\frac kn}\, n^{2-3p/2 } 
    \quad\quad[\text{by $k\le n+k$ and $n+k\ge
   n$}] \\
 &=& o\Big(\, \sqrt{\frac kn}\,\Big)   \qquad\quad \quad (n\rightarrow \infty)
\]
since $p>4/3$. To bound $U_{4,2}$, we again use Lemma \ref{Lem vBE for U},
but with $n+k$ in place of $n$ and with $f_{ij} := f_2$ for $j>n$
and $f_{ij} := 0$ for $j\le n$, to see that 
\[
 M_m &:=& \sum_{j=n+1}^m\sum_{i=1}^{j-1} f_2(X_i,X_j) 
      \qquad (m\in\{n+1,\ldots,n+k\})
\]
defines a martingal. Hence  Doob's inequality, \eqref{vBE for U},  
$c_2 := 4 \epsilon^{-p}  \E|f_2(x_1,X_2)|^p <\infty $, and $p>3/2$ yield 
\[
 \P\Big( |U_{4,2}| > \epsilon  \sqrt{\frac kn}\,     \Big)
 &\le& \P( \frac 1{n^{3/2}} \max_{m=n+1}^{n+ k}|M_m|>\epsilon\sqrt{\frac kn})\\
 &\le& (\epsilon n \sqrt{k})^{-p} \,\E  |M_{n+k}|^p \\
 &\le &    (\epsilon n \sqrt{k})^{-p}  
  4 \sum_{j=n+1}^{n+k}\sum_{i=1}^{j-1} \E |f_2(X_i,X_j)|^p  \\
 &\le& c_2 \cdot (n \sqrt{k})^{-p}k (n+k) \\   
 &=& c_2  \sqrt{\frac kn} \Big( k^{\frac{1-p}2}_{}n^{\frac{3}{2}-p}_{}
        +  k^{\frac {3 -p}2}_{}n^{\frac{1}{2}-p}_{}  \Big) \\
 &\le& c_2 \sqrt{\frac kn} \big( n^{\frac{3}{2}-p}_{}+\frac{k^{\frac34}}{n}\big)\\
 &=& o\Big(\sqrt{\frac kn}\,\Big)  \qquad\quad (\frac kn\rightarrow 0) 
\]

Investigation of $U_5$: Using \eqref{Rm strongly small}
with $t=\sqrt{k}$, we get 
\la                          \label{Bounding U_5}
 \P\Big( |U_5| > 2\epsilon  \sqrt{\frac kn} \, \Big) 
    &\le& \sum_{m=n}^{n+k}\P\Big(|R_m|>\epsilon \sqrt{\frac km} \,\Big) \\
 &=& \frac1{\sqrt{k}} \sum_{m=n}^{n+k}\sqrt{k} 
   \P\Big(|R_m|>\epsilon \sqrt{\frac km} \,\Big) \nonumber \\
 &=& \frac{k+1}{\sqrt{k}} \,o(\frac 1{\sqrt{n}})\nonumber\\
 &=& o\Big( \sqrt{\frac kn}\, \Big)   \qquad \quad (n\rightarrow \infty) \nonumber
\al

Combining the results for $U_1,\ldots,U_5$, we get 
\[
 P( |U|>  8 \epsilon\sqrt{k/n}\big) = o\Big(   \sqrt{\frac kn}\,\Big)
   \qquad \quad ( \frac kn \rightarrow 0)
\]%
\end{proof} 
The following lemma, which we have just used above when handling $U_4$,
is in principle  well known, see for example 
Koroljuk and Borovskich' book  
\cite[p.~72, Theorem 2.1.3, the case $r=c=2$]{Koroljuk.Bor.1994}
for the special case where the $f_{ij}$ are symmetric and independent
of $(i,j)$.
\begin{lem}\label{Lem vBE for U}
Let $X_1,\ldots,X_n$ be  independent $\cX$-valued
random variables and let $f_{ij}:\cX^2\rightarrow \R$ be measurable
with $\E |f_{ij}(X_i,X_j)| <\infty$ and with
$ \E \big(f_{ij}(X_i,X_j) {\pmb|}X_i \big)=0 $
for $1\le i<j\le n$. Then 
\[
  M_m &:=& \sum_{j=2}^m\sum_{i=1}^{j-1} f_{ij}(X_i,X_j) \qquad \quad
   (m\in \{2,\ldots,n \})
\]
defines a martingale, and for $p\in [1,2]$ we  have 
\la   \label{vBE for U}
   \E  |M_n|^p &\le & 4\,\sum_{j=2}^n\sum_{i=1}^{j-1} \E |f_{ij}(X_i,X_j)|^p
\al
\end{lem}

\begin{proof} Clearly $(M_m:m\in\{2,\ldots,n\})$ is a martingale
with respect to the $\sigma$-algebras $\sigma(X_1,\ldots,X_m)$,
and so is $(\sum_{i=1}^{m}f_{ij}(X_i,X_j) :  m\in\{2,\ldots,j-1\})$  
for every $j$.  Applying twice the inequality
of von Bahr and Esseen \cite[Theorem~2]{vBE.1965}   yields 
\[
    \E  |M_n|^p&\le&  
  2\,\sum_{j=2}^n  \E \Big|\sum_{i=1}^{j-1} f_{ij}(X_i,X_j)\Big|^p 
  \,\,\le\,\, 2\,\sum_{j=2}^n 2\,\sum_{i=1}^{j-1}   \E |f_{ij}(X_i,X_j)|^p
\]
\end{proof}

\subsection{Proof of Theorem \ref{Main result}} \label{Main proof.}
Lemma \ref{Main Lemma} combined with the elementary inequality  
$|\max a_m -\max b_m| \le \max |b_m -a_m|$ yields  
\la                \label{MainReductionSecondCase}
&& \max_{m=n}^{n+k} T_m  \\
\nonumber  &=& T_n +\frac 1{\sqrt{n}} \max_{m=n}^{n+k} \sum_{i=n+1}^mf_0(X_i)
      + \sqrt{\frac kn} \epsilon^{}_{\P}( \sqrt{\frac kn})  
  \qquad\quad (
      \frac kn\rightarrow 0)
\al
Let $\epsilon \in {]0,1]}$.   
In this proof, $c_1,c_2,c_3 \in \,]0,\infty[$ 
and implied constants in $O(\ldots)$-statements
do  not depend on $n,k,\epsilon$, 
but may  depend on $\alpha$, the law of $X_1$, and the sequence
$(T_n)$.    
Using first \eqref{MainReductionSecondCase} 
and then the independence and stationarity of the sequence
$(X_i)$, we get 
\la          \label{1-alpha n k}
 && 1-\alpha_{n,k}  \\ 
 \nonumber &=& \P(\max_{m=n}^{n+k} T_m \le 0) \\
\nonumber &\begin{array}{c}\le\\ \ge\end{array} & 
  \P\Big(\, T_n \le -  \frac 1{\sqrt{n}} \max_{m=n}^{n+k} \sum_{i=n+1}^mf_0(X_i)
                     \pm \epsilon\,\sqrt{\frac kn} \,\,\Big)  
    \,+\, o\Big( \sqrt{\frac kn}\,\Big)     \\
\nonumber &=& \int_\R \P( T_n \le y) \d Q_{n,k}(y)  
  \,+\, o\Big( \sqrt{\frac kn}\,\Big) 
\al
for $ k/n\rightarrow 0$,
with $Q_{n,k}$ denoting the law of 
\[
 Y_{n,k} &:=& -  \frac 1{\sqrt{n}} \max_{\ell=0 }^{k} S_\ell
                     \,\pm\, \epsilon\,\sqrt{\frac kn} 
 \qquad \quad (n,k\in \N)
\]
where
\[
  S_\ell &:=& \sum_{i=1}^\ell f_0(X_i)  \qquad\quad(\ell\in \N_0)
\]
Since the $f_0(X_i)$ are i.i.d., we can apply a result 
of Kac, see   \cite[Theorem 4.1]{Kac.1954}
and also \cite[ p.~330]{Spitzer.1956}, to get 
\la        \label{KacFormulaApplied}
 \E Y_{n,k} &=& -\frac 1{\sqrt{n}} \sum_{\ell=1}^k\frac 1{\ell}
              \E (S_\ell)_+      \,\pm\,\epsilon\,\sqrt{\frac kn} 
\al

Since $-Y_{n,k} = \max_{\ell =0}^k(\frac 1{\sqrt{n}} 
S_\ell\mp\epsilon\,\sqrt{k/n} ) $
is the maximum of a martingale, the usual 
 L$^2$  maximum inequality, see e.g. 
\cite[p.~248]{Durrett.2005}, yields
\la           \label{Ynk small 1}
\E Y_{n,k}^2 &\le& 4\,\E \Big(\frac 1{\sqrt{n}}S_k 
   \mp\epsilon\,\sqrt{\frac kn}  \Big)^2
 \,\, =\,\, 4(1 + \epsilon^2)\frac kn
 \,\, \le \,\, c_1 \frac kn
\al
and hence in particular, by the Chebyshev and Lyapunov inequalities,
\la  \label{Ynk small 2}
  \P( |Y_{n,k}| \ge 1)  \,\, \le \,\,  c_1 \frac kn, & &
| \E Y_{n,k}|  \,\, \le \,\, \sqrt{c_1}\sqrt{\frac kn} 
\al

An application of \cite[p.~343, Corollary 10.3.8, 
with $S_n= \mu(P) +T_n/\sqrt{n}$, 
$P$ the law of $X_1$,  $\beta=1-\alpha$, $\mu(P)=0$, 
$\sigma(P)=1$, $N_\beta=\Phi^{-1}(1-\alpha)$, $g_1=\overline{g}_1=0$, 
$B_0(g_1)=B_0(0)=0$, $P(f_0(\cdot,P)g_1)=0$,
and with $U_n$ in place of $V_n$ according to 
Discussion \ref{Remark on assumptions}]{PW.II}   yields   
\la   \label{AsmpPfW.a}
 && \P(T_n \le y) \\
 \nonumber &=& F_n(y) \,+\,
  o\big( \frac1{\sqrt{n}}\big)
  \qquad\quad  (n\rightarrow\infty, \text{ locally uniformly in }y\in\R)
\al
with
\la          \label{AsmpPfW.b}
 F_n(y)  &:=&\Phi\Big(\,\Phi^{-1}(1-\alpha) +  y
      + \frac {ay + by^2} {\sqrt{n}} \,\Big) 
     \qquad\quad  (n\in\N,\, y\in \R)
\al
where  $a,b \in \R$ depend only on $P,f_0,f_1, f_2$.

Since the functions $\P(T_n\le \cdot)$ and $F_n$ are $[0,1]$-valued,
we get
\la           \label{DiffOfIntegrals}
 &&\left|\int_\R\P(T_n\le y)\d Q_{n,k}(y)-\int_\R F_n(y)\d Q_{n,k}(y)\right|\\
\nonumber &\le& \int_{[-1,1]}\left|\P(T_n\le y)-F_n(y)\right| \d Q_{n,k}(y) 
      \,+\, Q_{n,k}(\R\setminus [-1,1]) \\
\nonumber &=&  o\big( \frac1{\sqrt{n}}\big) \,+\, O(\frac kn)  
     \qquad\qquad  (n\rightarrow\infty) 
\qquad [\text{by \eqref{AsmpPfW.a} and  \eqref{Ynk small 2}}] \\
\nonumber&=& o\Big(\,\sqrt{\frac kn} \,\Big)\qquad\qquad
\qquad\quad(\frac kn\rightarrow 0)
\al

Using 
\[
 F^{}_n(0) = 1-\alpha,&& F_n'(0) = \phi(\Phi^{-1}(1-\alpha)) 
  \cdot \big( 1+ \frac {a}{\sqrt{n}}\big)  
\]
\[
  \|F_n'' \|^{}_\infty &:=& \sup_{y\in\R} |F_n''(y)| \,\,\le\,\, c_2
       \qquad \quad (n\in \N)
\]
with $c_2$ depending only on $a$ and $b$, a Taylor expansion 
of $F_n$ around zero yields
\la   \label{IntFdQ}
 && \int_\R F_n(y)\d Q_{n,k}(y)\\
 \nonumber 
&=&F_n(0) \,+\, F_n'(0) \E Y_{n,k}\,+\, O(\|F_n'' \|^{}_\infty \E Y^2_{n,k} )\\
\nonumber
&=& 1-\alpha \,+\, \phi(\Phi^{-1}(1-\alpha)) 
  \big( 1 + \frac a{\sqrt{n}}\big) \E Y_{n,k} 
  \,+\, O\big(\frac kn \big) \\
&=& \nonumber 1-\alpha \,-\, 
  \frac{\phi(\Phi^{-1}(1-\alpha))}{\sqrt{n}}\sum_{\ell=1}^k\frac 1\ell
      \E (S_\ell)^{}_+ 
        \,\pm\, c_3 \epsilon\,\sqrt{\frac kn}  
  \,+ \,  O\big(\frac kn \big)
\al
using \eqref{Ynk small 1} and  \eqref{KacFormulaApplied}. 
Combining \eqref{Def alpha n k}, \eqref{1-alpha n k},
\eqref{DiffOfIntegrals} and \eqref{IntFdQ} yields \eqref{MainResultII}.

We have $\E (S_\ell/\sqrt{\ell})^{}_+ \rightarrow 1/\sqrt{2\pi}$
for $\ell \rightarrow \infty$,
by the uniform integrability of $(S_\ell/\sqrt{\ell})^{}_+$
following from $\E (S_\ell/\sqrt{\ell})^{2}_+\le 1$ and by the central limit
theorem, compare 
\label{could have used von Bahr}
\cite[Theorems~25.12 and 27.1]{Billingsley.1995}.
It follows that 
$\sum_{\ell =1}^k \frac 1\ell \E(S_\ell)^{}_+
 \sim \frac{1}{\sqrt{2\pi}} \sum_{\ell =1}^k\frac 1{\sqrt{\ell}}
\sim \frac{2}{\sqrt{2\pi}}\sqrt{k} $ for $k\rightarrow \infty$. 
Hence \eqref{MainResultII} yields \eqref{MainResultIII}. \Halmos

\subsection{Proof of Lemma \ref{Lemma on h}}\label{Proof on h}
We have
\la  \label{rephcom}
  h&=& \frac 1{\sqrt{2\pi}} \cdot
  \big(]0,\infty[\,\ni x\mapsto \frac 1x\big)
 \circ \frac {1-\Phi}{\phi} 
 \circ \big(]0,1[\,\ni \alpha\mapsto \Phi^{-1}(1-\alpha)\big)
\al
Since Mills' ratio 
$\R \ni x\mapsto \frac{1-\Phi}{\phi}(x)=\int_0^\infty \exp(-xt-t^2/2)\d t$
and the other two composition factors   in \eqref{rephcom}
are strictly decreasing, so is $h$.
Applying the  well-known asymptotics 
$ \frac{1-\Phi}{\phi}(x) \sim \frac 1x$ for $x\rightarrow \infty$
and $\Phi^{-1}(1-\alpha) \sim \sqrt{2\log(\frac 1\alpha)}$ for
$\alpha\rightarrow 0$ to \eqref{rephcom}, we get \eqref{asymp h}.
\Halmos

\subsection{Proof of Lemma \ref{Lemma on Rn}}\label{Proof on Rn}
By assumption 
$
  \lim_{y\rightarrow\infty}\sup_{n\ge n_0}  
 \E Y_n  (Y_n >y ) =0
$, and for  $t\ge 1$
\[
 t\sqrt{n}\, \P\Big(|R_n|>\frac{t\epsilon}{\sqrt{n}}\Big) 
 &\le& \epsilon^{-p}t^{1-p} n^{\frac{1+p}2}
 \E \Big( |R_n|^p(|R_n|>\frac{t\epsilon}{\sqrt{n}})\Big)  \\
 &\le& \epsilon^{-p}
\E\Big( Y_n\,( Y_n >\epsilon^p  \sqrt{n} ) \Big) 
\] 
\hfill\Halmos

\section*{Acknowledgements}
Thanks to Peter Hall for some helpful comments on a previous 
version of this paper.

\end{document}